\pdfoutput=1
\documentclass[10pt,conference]{IEEEtran}
\usepackage{pgfplots}
\usepackage[hyphens]{url} 
\usepackage[english]{babel}
\usepackage{fixltx2e}

\usepackage[utf8]{inputenc}
\usepackage[T1]{fontenc}

\usepackage{multicol}

\usepackage{mathpazo}
\usepackage[scaled=0.95]{helvet}
\usepackage{courier}
\usepackage{graphicx}
\usepackage{comment}

\usepackage{fancyhdr} 

\usepackage{ifthen}
\usepackage{amssymb,amsmath,amsthm,stmaryrd,mathrsfs,wasysym}
\usepackage{enumitem,mathtools,xspace,xcolor}
\definecolor{darkgreen}{rgb}{0,0.45,0}
\usepackage{aliascnt}
\usepackage[style=numeric,maxnames=10,backend=bibtex,%
            hyperref=true]{biblatex}
\usepackage{braket} 
\usepackage{tikz-cd}
\usepackage{tikz}
\usetikzlibrary{decorations.pathmorphing}
\usepackage[inference]{semantic}
\usepackage{booktabs}

\usepackage[colorlinks,citecolor=darkgreen,linkcolor=darkgreen,unicode]{hyperref}

\usepackage[capitalize]{cleveref}

\makeatletter
\newcommand{\sectionNotes}{\phantomsection\section*{Notes}\addcontentsline{toc}{section}{Notes}\markright{\textsc{\@chapapp{} \thechapter{} Notes}}}
\newcommand{\sectionExercises}[1]{\phantomsection\section*{Exercises}\addcontentsline{toc}{section}{Exercises}\markright{\textsc{\@chapapp{} \thechapter{} Exercises}}}
\makeatother

\newcommand{\jdeq}{\equiv}      

\newcommand{\defeq}{\vcentcolon\equiv}  

\newcommand{\define}[1]{\textbf{#1}}

\makeatletter
\def\prd#1{\@ifnextchar\bgroup{\prd@parens{#1}}{\@ifnextchar\sm{\prd@parens{#1}\@eatsm}{\prd@noparens{#1}}}}
\def\prd@parens#1{\@ifnextchar\bgroup%
  {\mathchoice{\@dprd{#1}}{\@tprd{#1}}{\@tprd{#1}}{\@tprd{#1}}\prd@parens}%
  {\@ifnextchar\sm%
    {\mathchoice{\@dprd{#1}}{\@tprd{#1}}{\@tprd{#1}}{\@tprd{#1}}\@eatsm}%
    {\mathchoice{\@dprd{#1}}{\@tprd{#1}}{\@tprd{#1}}{\@tprd{#1}}}}}
\def\@eatsm\sm{\sm@parens}
\def\prd@noparens#1{\mathchoice{\@dprd@noparens{#1}}{\@tprd{#1}}{\@tprd{#1}}{\@tprd{#1}}}
\def\lprd#1{\@ifnextchar\bgroup{\@lprd{#1}\lprd}{\@@lprd{#1}}}
\def\@lprd#1{\mathchoice{{\textstyle\prod}}{\prod}{\prod}{\prod}({\textstyle #1})\;}
\def\@@lprd#1{\mathchoice{{\textstyle\prod}}{\prod}{\prod}{\prod}({\textstyle #1}),\ }
\def\tprd#1{\@tprd{#1}\@ifnextchar\bgroup{\tprd}{}}
\def\@tprd#1{\mathchoice{{\textstyle\prod_{(#1)}}}{\prod_{(#1)}}{\prod_{(#1)}}{\prod_{(#1)}}}
\def\dprd#1{\@dprd{#1}\@ifnextchar\bgroup{\dprd}{}}
\def\@dprd#1{\prod_{(#1)}\,}
\def\@dprd@noparens#1{\prod_{#1}\,}

\def\lam#1{{\lambda}\@lamarg#1:\@endlamarg\@ifnextchar\bgroup{.\,\lam}{.\,}}
\def\@lamarg#1:#2\@endlamarg{\if\relax\detokenize{#2}\relax #1\else\@lamvar{\@lameatcolon#2},#1\@endlamvar\fi}
\def\@lamvar#1,#2\@endlamvar{(#2\,{:}\,#1)}
\def\@lameatcolon#1:{#1}

\def\lamu#1{{\lambda}\@lamuarg#1:\@endlamuarg\@ifnextchar\bgroup{.\,\lamu}{.\,}}
\def\@lamuarg#1:#2\@endlamuarg{#1}

\def\fall#1{\forall (#1)\@ifnextchar\bgroup{.\,\fall}{.\,}}

\def\exis#1{\exists (#1)\@ifnextchar\bgroup{.\,\exis}{.\,}}

\def\sm#1{\@ifnextchar\bgroup{\sm@parens{#1}}{\@ifnextchar\prd{\sm@parens{#1}\@eatprd}{\sm@noparens{#1}}}}
\def\sm@parens#1{\@ifnextchar\bgroup%
  {\mathchoice{\@dsm{#1}}{\@tsm{#1}}{\@tsm{#1}}{\@tsm{#1}}\sm@parens}%
  {\@ifnextchar\prd%
    {\mathchoice{\@dsm{#1}}{\@tsm{#1}}{\@tsm{#1}}{\@tsm{#1}}\@eatprd}%
    {\mathchoice{\@dsm{#1}}{\@tsm{#1}}{\@tsm{#1}}{\@tsm{#1}}}}}
\def\@eatprd\prd{\prd@parens}
\def\sm@noparens#1{\mathchoice{\@dsm@noparens{#1}}{\@tsm{#1}}{\@tsm{#1}}{\@tsm{#1}}}
\def\lsm#1{\@ifnextchar\bgroup{\@lsm{#1}\lsm}{\@@lsm{#1}}}
\def\@lsm#1{\mathchoice{{\textstyle\sum}}{\sum}{\sum}{\sum}({\textstyle #1})\;}
\def\@@lsm#1{\mathchoice{{\textstyle\sum}}{\sum}{\sum}{\sum}({\textstyle #1}),\ }
\def\tsm#1{\@tsm{#1}\@ifnextchar\bgroup{\tsm}{}}
\def\@tsm#1{\mathchoice{{\textstyle\sum_{(#1)}}}{\sum_{(#1)}}{\sum_{(#1)}}{\sum_{(#1)}}}
\def\dsm#1{\@dsm{#1}\@ifnextchar\bgroup{\dsm}{}}
\def\@dsm#1{\sum_{(#1)}\,}
\def\@dsm@noparens#1{\sum_{#1}\,}

\def\wtype#1{\@ifnextchar\bgroup%
  {\mathchoice{\@twtype{#1}}{\@twtype{#1}}{\@twtype{#1}}{\@twtype{#1}}\wtype}%
  {\mathchoice{\@twtype{#1}}{\@twtype{#1}}{\@twtype{#1}}{\@twtype{#1}}}}
\def\lwtype#1{\@ifnextchar\bgroup{\@lwtype{#1}\lwtype}{\@@lwtype{#1}}}
\def\@lwtype#1{\mathchoice{{\textstyle\mathsf{W}}}{\mathsf{W}}{\mathsf{W}}{\mathsf{W}}({\textstyle #1})\;}
\def\@@lwtype#1{\mathchoice{{\textstyle\mathsf{W}}}{\mathsf{W}}{\mathsf{W}}{\mathsf{W}}({\textstyle #1}),\ }
\def\twtype#1{\@twtype{#1}\@ifnextchar\bgroup{\twtype}{}}
\def\@twtype#1{\mathchoice{{\textstyle\mathsf{W}_{(#1)}}}{\mathsf{W}_{(#1)}}{\mathsf{W}_{(#1)}}{\mathsf{W}_{(#1)}}}
\def\dwtype#1{\@dwtype{#1}\@ifnextchar\bgroup{\dwtype}{}}
\def\@dwtype#1{\mathsf{W}_{(#1)}\,}

\def\wtypeh#1{\@ifnextchar\bgroup%
  {\mathchoice{\@lwtypeh{#1}}{\@twtypeh{#1}}{\@twtypeh{#1}}{\@twtypeh{#1}}\wtypeh}%
  {\mathchoice{\@@lwtypeh{#1}}{\@twtypeh{#1}}{\@twtypeh{#1}}{\@twtypeh{#1}}}}
\def\lwtypeh#1{\@ifnextchar\bgroup{\@lwtypeh{#1}\lwtypeh}{\@@lwtypeh{#1}}}
\def\@lwtypeh#1{\mathchoice{{\textstyle\mathsf{W}^h}}{\mathsf{W}^h}{\mathsf{W}^h}{\mathsf{W}^h}({\textstyle #1})\;}
\def\@@lwtypeh#1{\mathchoice{{\textstyle\mathsf{W}^h}}{\mathsf{W}^h}{\mathsf{W}^h}{\mathsf{W}^h}({\textstyle #1}),\ }
\def\twtypeh#1{\@twtypeh{#1}\@ifnextchar\bgroup{\twtypeh}{}}
\def\@twtypeh#1{\mathchoice{{\textstyle\mathsf{W}^h_{(#1)}}}{\mathsf{W}^h_{(#1)}}{\mathsf{W}^h_{(#1)}}{\mathsf{W}^h_{(#1)}}}
\def\dwtypeh#1{\@dwtypeh{#1}\@ifnextchar\bgroup{\dwtypeh}{}}
\def\@dwtypeh#1{\mathsf{W}^h_{(#1)}\,}

\makeatother

\newcommand{\proj}[1]{\ensuremath{\mathsf{pr}_{#1}}\xspace}

\newcommand{\pairr}[1]{{\mathopen{}(#1)\mathclose{}}}

\newcommand{\idsym}{{=}}
\newcommand{\id}[3][]{\ensuremath{#2 =_{#1} #3}\xspace}

\newcommand{\refl}[1]{\ensuremath{\mathsf{refl}_{#1}}\xspace}

\newcommand{\ct}{%
  \mathchoice{\mathbin{\raisebox{0.5ex}{$\displaystyle\centerdot$}}}%
             {\mathbin{\raisebox{0.5ex}{$\centerdot$}}}%
             {\mathbin{\raisebox{0.25ex}{$\scriptstyle\,\centerdot\,$}}}%
             {\mathbin{\raisebox{0.1ex}{$\scriptscriptstyle\,\centerdot\,$}}}
}

\newcommand{\trans}[2]{\ensuremath{{#1}_{*}\mathopen{}\left({#2}\right)\mathclose{}}\xspace}

\newcommand{\mapfunc}[1]{\ensuremath{\mathsf{ap}_{#1}}\xspace} 

\newcommand{\mapdepfunc}[1]{\ensuremath{\mathsf{apd}_{#1}}\xspace} 

\newcommand{\idfunc}[1][]{\ensuremath{\mathsf{id}_{#1}}\xspace}

\newcommand{\htpy}{\sim}

\newcommand{\eqv}[2]{\ensuremath{#1 \simeq #2}\xspace}

\newcommand{\eqvsym}{\simeq}    

\newcommand{\hfib}[2]{{\mathsf{fib}}_{#1}(#2)}

\newcommand{\UU}{\ensuremath{\mathcal{U}}\xspace}

\let\type\UU

\newcommand{\emptyt}{\ensuremath{\mathbf{0}}\xspace}

\newcommand{\unit}{\ensuremath{\mathbf{1}}\xspace}
\newcommand{\ttt}{\ensuremath{\star}\xspace}

\newcommand{\inlsym}{{\mathsf{inl}}}
\newcommand{\inrsym}{{\mathsf{inr}}}
\newcommand{\inl}{\ensuremath\inlsym\xspace}
\newcommand{\inr}{\ensuremath\inrsym\xspace}

\newcommand{\glue}{\mathsf{glue}}

\newcommand{\Sn}{\mathbb{S}}
\newcommand{\base}{\ensuremath{\mathsf{base}}\xspace}

\newcommand{\susp}{\Sigma}
\newcommand{\north}{\mathsf{N}}
\newcommand{\south}{\mathsf{S}}

\newcommand{\blank}{\mathord{\hspace{1pt}\text{--}\hspace{1pt}}}

\newcommand{\nameless}{\mathord{\hspace{1pt}\underline{\hspace{1ex}}\hspace{1pt}}}

\newcommand{\N}{\ensuremath{\mathbb{N}}\xspace}

\newcommand{\Z}{\ensuremath{\mathbb{Z}}\xspace}

\makeatletter
\def\defthm#1#2#3{%
  \newaliascnt{#1}{thm}
  \newtheorem{#1}[#1]{#2}
  \aliascntresetthe{#1}
  \crefname{#1}{#2}{#3}}

\newtheorem{thm}{Theorem}[section]
\crefname{thm}{Theorem}{Theorems}
\let\xx@thm\@thm
\AtBeginDocument{\let\@thm\xx@thm}
\defthm{cor}{Corollary}{Corollaries}
\defthm{lem}{Lemma}{Lemmas}
\defthm{axiom}{Axiom}{Axioms}
\theoremstyle{definition}
\defthm{defn}{Definition}{Definitions}
\theoremstyle{remark}
\defthm{rmk}{Remark}{Remarks}
\defthm{eg}{Example}{Examples}
\defthm{egs}{Examples}{Examples}
\defthm{notes}{Notes}{Notes}

\crefformat{section}{\S#2#1#3}
\Crefformat{section}{Section~#2#1#3}
\crefrangeformat{section}{\S\S#3#1#4--#5#2#6}
\Crefrangeformat{section}{Sections~#3#1#4--#5#2#6}
\crefmultiformat{section}{\S\S#2#1#3}{ and~#2#1#3}{, #2#1#3}{ and~#2#1#3}
\Crefmultiformat{section}{Sections~#2#1#3}{ and~#2#1#3}{, #2#1#3}{ and~#2#1#3}
\crefrangemultiformat{section}{\S\S#3#1#4--#5#2#6}{ and~#3#1#4--#5#2#6}{, #3#1#4--#5#2#6}{ and~#3#1#4--#5#2#6}
\Crefrangemultiformat{section}{Sections~#3#1#4--#5#2#6}{ and~#3#1#4--#5#2#6}{, #3#1#4--#5#2#6}{ and~#3#1#4--#5#2#6}

\crefformat{appendix}{Appendix~#2#1#3}
\Crefformat{appendix}{Appendix~#2#1#3}
\crefrangeformat{appendix}{Appendices~#3#1#4--#5#2#6}
\Crefrangeformat{appendix}{Appendices~#3#1#4--#5#2#6}
\crefmultiformat{appendix}{Appendices~#2#1#3}{ and~#2#1#3}{, #2#1#3}{ and~#2#1#3}
\Crefmultiformat{appendix}{Appendices~#2#1#3}{ and~#2#1#3}{, #2#1#3}{ and~#2#1#3}
\crefrangemultiformat{appendix}{Appendices~#3#1#4--#5#2#6}{ and~#3#1#4--#5#2#6}{, #3#1#4--#5#2#6}{ and~#3#1#4--#5#2#6}
\Crefrangemultiformat{appendix}{Appendices~#3#1#4--#5#2#6}{ and~#3#1#4--#5#2#6}{, #3#1#4--#5#2#6}{ and~#3#1#4--#5#2#6}

\crefname{part}{Part}{Parts}

\crefformat{paragraph}{\S#2#1#3}
\Crefformat{paragraph}{Paragraph~#2#1#3}
\crefrangeformat{paragraph}{\S\S#3#1#4--#5#2#6}
\Crefrangeformat{paragraph}{Paragraphs~#3#1#4--#5#2#6}
\crefmultiformat{paragraph}{\S\S#2#1#3}{ and~#2#1#3}{, #2#1#3}{ and~#2#1#3}
\Crefmultiformat{paragraph}{Paragraphs~#2#1#3}{ and~#2#1#3}{, #2#1#3}{ and~#2#1#3}
\crefrangemultiformat{paragraph}{\S\S#3#1#4--#5#2#6}{ and~#3#1#4--#5#2#6}{, #3#1#4--#5#2#6}{ and~#3#1#4--#5#2#6}
\Crefrangemultiformat{paragraph}{Paragraphs~#3#1#4--#5#2#6}{ and~#3#1#4--#5#2#6}{, #3#1#4--#5#2#6}{ and~#3#1#4--#5#2#6}

\crefname{figure}{Figure}{Figures}

\let\autoref\cref

\let\c@equation\c@thm
\numberwithin{equation}{section}

\setitemize[1]{leftmargin=2em}
\setenumerate[1]{leftmargin=*}

\setitemize[1]{itemsep=-0.2em}
\setenumerate[1]{itemsep=-0.2em}

\def\noteson{%
\gdef\note##1{\mbox{}\marginpar{\color{blue}\textasteriskcentered\ ##1}}}

\noteson

\newcounter{symindex}

\makeatletter

\renewcommand{\id}[3][]{
  \@ifnextchar\bgroup
    {#2 \mathbin{\idsym_{#1}} \id[#1]{#3}}
    {#2 \mathbin{\idsym_{#1}} #3}
  }

\renewcommand{\eqv}[2]{
  \@ifnextchar\bgroup
    {#1 \eqvsym \eqv{#2}}
    {#1 \eqvsym #2}
  }

\newcommand{\ctsym}{%
  \mathchoice{\mathbin{\raisebox{0.5ex}{$\displaystyle\centerdot$}}}%
             {\mathbin{\raisebox{0.5ex}{$\centerdot$}}}%
             {\mathbin{\raisebox{0.25ex}{$\scriptstyle\,\centerdot\,$}}}%
             {\mathbin{\raisebox{0.1ex}{$\scriptscriptstyle\,\centerdot\,$}}}
  }

\renewcommand{\ct}[3][]{
  \@ifnextchar\bgroup
    {#2 \mathbin{\ctsym_{#1}} \ct[#1]{#3}}
    {#2 \mathbin{\ctsym_{#1}} #3}
  }

\makeatother

\makeatletter
\renewcommand{\@dprd}{\@tprd}
\renewcommand{\@dsm}{\@tsm}
\renewcommand{\@dprd@noparens}{\@tprd}
\renewcommand{\@dsm@noparens}{\@tsm}

\renewcommand{\@tprd}[1]{\mathchoice{{\textstyle\prod_{(#1)}\,}}{\prod_{(#1)}\,}{\prod_{(#1)}\,}{\prod_{(#1)}\,}}
\renewcommand{\@tsm}[1]{\mathchoice{{\textstyle\sum_{(#1)}\,}}{\sum_{(#1)}\,}{\sum_{(#1)}\,}{\sum_{(#1)}\,}}

\newcommand{\implicitargumentson}{\boolean{true}}

\newcommand{\@ifnextchar@starorbrace}[2]
  {\@ifnextchar*{#1}{\@ifnextchar\bgroup{#1}{#2}}}
  
\renewcommand{\prd}{\@ifnextchar*{\@iprd}{\@prd}}

\newcommand{\@prd}[1]
  {\@ifnextchar@starorbrace
    {\prd@parens{#1}}
    {\@ifnextchar\sm{\prd@parens{#1}\@eatsm}{\prd@noparens{#1}}}}
\newcommand{\@prd@parens}{\@ifnextchar*{\@iprd}{\prd@parens}}
\renewcommand{\prd@parens}[1]
  {\@ifnextchar@starorbrace
    {\@theprd{#1}\@prd@parens}
    {\@ifnextchar\sm{\@theprd{#1}\@eatsm}{\@theprd{#1}}}}
\newcommand{\@theprd}[1]
  {\mathchoice{\@dprd{#1}}{\@tprd{#1}}{\@tprd{#1}}{\@tprd{#1}}}
\renewcommand{\dprd}[1]{\@dprd{#1}\@ifnextchar@starorbrace{\dprd}{}}
\renewcommand{\tprd}[1]{\@tprd{#1}\@ifnextchar@starorbrace{\tprd}{}}

\newcommand{\@theiprd}[1]{\mathchoice{\@diprd{#1}}{\@tiprd{#1}}{\@tiprd{#1}}{\@tiprd{#1}}}
\newcommand{\@iprd}[2]{\@ifnextchar@starorbrace%
  {\@theiprd{#2}\@prd@parens}%
  {\@ifnextchar\sm%
    {\@theiprd{#2}\@eatsm}%
    {\@theiprd{#2}}}}
\def\@tiprd#1{
  \ifthenelse{\implicitargumentson}
    {\@@tiprd{#1}\@ifnextchar\bgroup{\@tiprd}{}}
    {\@tprd{#1}}}
\def\@@tiprd#1{\mathchoice{{\textstyle\prod_{\{#1\}}\,}}{\prod_{\{#1\}}\,}{\prod_{\{#1\}}\,}{\prod_{\{#1\}}\,}}
\def\@diprd{
  \ifthenelse{\implicitargumentson}
    {\@tiprd}
    {\@tprd}}

\def\@eatprd\prd{\@prd@parens}

\makeatother

\makeatletter
\def\tfall#1{\forall_{(#1)}\@ifnextchar\bgroup{\,\tfall}{\,}}
\renewcommand{\fall}{\tfall}

\def\texis#1{\exists_{(#1)}\@ifnextchar\bgroup{\,\texis}{\,}}
\renewcommand{\exis}{\texis}

\def\uexis#1{\exists!_{(#1)}\@ifnextchar\bgroup{\,\uexis}{\,}}
\makeatother

\newcommand{\typefont}{\mathsf} 
\newcommand{\catfont}{\mathrm} 

\renewcommand{\UU}{\typefont{U}}

\renewcommand{\pairr}[1]{{\mathopen{}\langle #1\rangle\mathclose{}}}
\renewcommand{\type}{\typefont{Type}}

\renewcommand{\susp}{\typefont{\Sigma}}

\newcommand{\tfcolim}{\typefont{colim}}

\newcommand{\sbrck}[1]{\Vert #1\Vert}

\makeatletter
\newcommand{\jctx}{\@ifnextchar*{\@jctxAlignTrue}{\@jctxAlignFalse}}
\newcommand{\@jctxAlignTrue}[2]{& \vdash #2~ctx}
\newcommand{\@jctxAlignFalse}[1]{\vdash #1~ctx}

\newcommand{\jtype}{\@ifnextchar*{\@jtypeAlignTrue}{\@jtypeAlignFalse}}
\newcommand{\@jtypeAlignFalse}[2]{#1\vdash #2~type}
\newcommand{\@jtypeAlignTrue}[3]{#2 & \vdash #3~type}

\newcommand{\jtermc}{\@ifnextchar*{\@jtermcAlignTrue}{\@jtermcAlignFalse}}
\newcommand{\@jtermcAlignTrue}[3]{ & \vdash #3:#2}
\newcommand{\@jtermcAlignFalse}[2]{\vdash #2:#1}

\newcommand{\jtermt}{\@ifnextchar*{\@jtermtAlignTrue}{\@jtermtAlignFalse}}
\newcommand{\@jtermtAlignTrue}[4]{#2 & \vdash #4:#3}
\newcommand{\@jtermtAlignFalse}[3]{#1 \vdash #3:#2}

\newcommand{\jctxeq}{\@ifnextchar*{\@jctxeqAlignTrue}{\@jctxeqAlignFalse}}
\newcommand{\@jctxeqAlignTrue}[3]{& \vdash #2\jdeq #3~ctx}
\newcommand{\@jctxeqAlignFalse}[2]{\vdash #1\jdeq #2~ctx}

\newcommand{\jtypeeq}{\@ifnextchar*{\@jtypeeqAlignTrue}{\@jtypeeqAlignFalse}}
\newcommand{\@jtypeeqAlignTrue}[4]{#2 & \vdash #3\jdeq #4~type}
\newcommand{\@jtypeeqAlignFalse}[3]{#1\vdash #2\jdeq #3~type}

\newcommand{\jtermceq}{\@ifnextchar*{\@jtermceqAlignTrue}{\@jtermceqAlignFalse}}
\newcommand{\@jtermceqAlignTrue}[4]{& \vdash #3\jdeq #4:#2}
\newcommand{\@jtermceqAlignFalse}[3]{\vdash #2\jdeq #3:#1}

\newcommand{\jtermteq}{\@ifnextchar*{\@jtermteqAlignTrue}{\@jtermteqAlignFalse}}
\newcommand{\@jtermteqAlignTrue}[5]{#2 & \vdash #4\jdeq #5:#3}
\newcommand{\@jtermteqAlignFalse}[4]{#1\vdash #3\jdeq #4:#2}
\makeatother

\makeatletter
\newcommand{\ctxext}[2]{\@ctxext@ctx #1.\@ctxext@type #2}
\newcommand{\@ctxext}{\@ifnextchar\bgroup{\@@ctxext}{}}
\newcommand{\@ctxext@ctx}{\@ifnextchar\ctxext{\@ctxext@nested}{\@ifnextchar\ctxwk{\@ctxwk@nested}{\@ctxext}}}
\newcommand{\@ctxext@type}{\@ifnextchar\ctxext{\@ctxext@nested}{\@ifnextchar\subst{\@subst@nested}{\@ctxext}}}
\newcommand{\@@ctxext}[1]{\@ifnextchar\bgroup{\@ctxext@parens{#1}}{#1}}
\newcommand{\@ctxext@parens}[2]{(\ctxext{#1}{#2})}
\newcommand{\@ctxext@nested}[3]{\@ctxext@parens{#2}{#3}}
\makeatother

\makeatletter
\newcommand{\subst}[2]{\@subst@type #2[\@subst@term #1]}
\newcommand{\@subst}{\@ifnextchar\bgroup{\@@subst}{}}
\newcommand{\@@subst}[1]{\@ifnextchar\bgroup{\subst{#1}}{#1}}
\newcommand{\@subst@term}{\@subst}
\newcommand{\@subst@type}{\@ifnextchar\ctxext{\@ctxext@nested}{\@ifnextchar\ctxwk{\@ctxwk@nested}{\@subst}}}
\newcommand{\@subst@nested}[3]{\@subst@parens{#2}{#3}}
\newcommand{\@subst@parens}[2]{(\subst{#1}{#2})}
\makeatother

\makeatletter
\newcommand{\ctxwk}[2]{\langle\@ctxwk@act #1\rangle\@ctxwk@pass #2}
\newcommand{\@ctxwk}{\@ifnextchar\bgroup{\@@ctxwk}{}}
\newcommand{\@@ctxwk}[1]{\@ifnextchar\bgroup{\ctxwk{#1}}{#1}}
\newcommand{\@ctxwk@act}{\@ctxwk}
\newcommand{\@ctxwk@pass}{\@ifnextchar\ctxext{\@ctxext@nested}{\@ifnextchar\subst{\@subst@nested}{\@ctxwk}}}
\newcommand{\@ctxwk@parens}[2]{(\ctxwk{#1}{#2})}
\newcommand{\@ctxwk@nested}[3]{\@ctxwk@parens{#2}{#3}}
\makeatother

\makeatletter
\newcommand{\pt}[1][]{*_{
  \@ifnextchar\undergraph{\@undergraph@nested}
    {\@ifnextchar\underovergraph{\@underovergraph@nested}{}}#1}}
\makeatother

\makeatletter
\newcommand{\pts}[1]{{\@graphop@nested{#1}}_{0}}
\newcommand{\edg}[1]{{\@graphop@nested{#1}}_{1}}
\newcommand{\@graphop@nested}[1]
  {\@ifnextchar\ctxext{\@ctxext@nested}
      {\@ifnextchar\undergraph{\@undergraph@nested}
         {\@ifnextchar\underovergraph{\@underovergraph@nested}{}}}
    #1}
\makeatother

\makeatletter
\newcommand{\@undergraphtest}[2]{\@ifnextchar({#1}{#2}}
\newcommand{\undergraph}[2]{\@undergraphtest{\@undergraph@parens{#1}{#2}}{\@undergraph{#1}{#2}}}
\newcommand{\@undergraph}[2]{{#2/#1}}
\newcommand{\@undergraph@nested}[3]{\@undergraph@parens{#2}{#3}}
\newcommand{\@undergraph@parens}[2]{(\@undergraph{#1}{#2})}
\makeatother

\makeatletter
\newcommand{\underovergraph}[2]{\@underovergraphtest{\@underovergraph@parens{#1}{#2}}{\@underovergraph{#1}{#2}}}
\newcommand{\@underovergraph}[2]{{#2}\,{\parallel}\,{#1}}
\newcommand{\@underovergraphtest}{\@undergraphtest}
\newcommand{\@underovergraph@parens}[2]{(\@underovergraph{#1}{#2})}
\newcommand{\@underovergraph@nested}[3]{\@underovergraph@parens{#2}{#3}}
\makeatother

\tikzset{patharrow/.style={double,double equal sign distance,-,font=\scriptsize}}
\tikzset{description/.style={fill=white,inner sep=2pt}}

\tikzset{commutative diagrams/column sep/Huge/.initial=18ex}

\defthm{conj}{Conjecture}{Conjectures}

\usepackage{balance}

\newcommand{\rprojective}[1]{\mathbb{R}\mathsf{P}^{#1}}

\newcommand{\inftyGpd}{\infty\catfont{Gpd}}
\newcommand{\Zmodtwo}{\Z/2\Z}

\makeatletter
\newcommand{\@ifnextcharamong}[2]
  {\@ifnextchar\bgroup{\@@ifnextchar{#1}{\@@ifnextcharamong{#1}{#2}}}{#2}}
\newcommand{\@@ifnextchar}[3]{\@ifnextchar{#3}{#1}{#2}}
\newcommand{\@@ifnextcharamong}[3]{\@ifnextcharamong{#1}{#2}}
\makeatother



\newcommand{\join}[2]{{#1}*{#2}}
\newcommand{\gluesym}{{\mathsf{glue}}}
\newcommand{\jglue}{\ensuremath\gluesym\xspace}

\title{The real projective spaces\\in homotopy type theory}
\date{\today}
\author{\IEEEauthorblockN{Ulrik Buchholtz}
  \IEEEauthorblockA{Technische Universit{\"a}t Darmstadt\\
  Email: buchholtz@mathematik.tu-darmstadt.de}
  \and
  \IEEEauthorblockN{Egbert Rijke}
  \IEEEauthorblockA{Carnegie Mellon University\\
  Email: erijke@andrew.cmu.edu}}

\begin{document}


\maketitle
\begin{abstract}
Homotopy type theory is a version of Martin-L{\"o}f type theory taking advantage of its homotopical models. In particular, we can use and construct objects of homotopy theory and reason about them using higher inductive types. In this article, we construct the \emph{real projective spaces}, key players in homotopy theory, as certain higher inductive types in homotopy type theory. The classical definition of $\mathbb{R}\mathrm{P}^n$, as the quotient space identifying antipodal points of the $n$\nobreakdash-sphere, does not translate directly to homotopy type theory. Instead, we define $\mathbb{R}\mathrm{P}^n$ by induction on $n$ simultaneously with its tautological bundle of $2$\nobreakdash-element sets. As the base case, we take $\mathbb{R}\mathrm{P}^{-1}$ to be the empty type. In the inductive step, we take $\mathbb{R}\mathrm{P}^{n+1}$ to be the mapping cone of the projection map of the tautological bundle of $\rprojective{n}$, and we use its universal property and the univalence axiom to define the tautological bundle on $\mathbb{R}\mathrm{P}^{n+1}$. 

By showing that the total space of the tautological bundle of $\mathbb{R}\mathbb{P}^n$ is the $n$\nobreakdash-sphere $\Sn^n$, we retrieve the classical description of $\mathbb{R}\mathrm{P}^{n+1}$ as $\mathbb{R}\mathrm{P}^n$ with an $(n+1)$\nobreakdash-disk attached to it. The infinite dimensional real projective space $\mathbb{R}\mathrm{P}^\infty$, defined as the sequential colimit of $\mathbb{R}\mathrm{P}^n$ with the canonical inclusion maps, is equivalent to the \emph{Eilenberg-MacLane space} $K(\Zmodtwo,1)$, which here arises as the subtype of the universe consisting of $2$\nobreakdash-element types. Indeed, the infinite dimensional projective space classifies the $0$\nobreakdash-sphere bundles, which one can think of as synthetic line bundles.

These constructions in homotopy type theory further illustrate the utility of homotopy type theory, including the interplay of type theoretic and homotopy theoretic ideas.

\end{abstract}

\section{Introduction}

Homotopy type theory emerged from the discovery of a homotopical interpretation of Martin-L{\"o}f's identity types \cite{AwodeyWarren2009} and in particular from the construction of the model in simplicial sets \cite{AfterVoevodsky} and the proposal of the univalence axiom \cite{Voevodsky06,Voevodsky10}. We refer to the book on homotopy type theory for details \cite{TheBook}.

Homotopy type theory allows us to reason synthetically about the objects of algebraic topology (spaces, paths, homotopies, etc.) analogously to how the setting of Euclidean geometry allows us to reason synthetically about points, lines, circles, etc.\ (as opposed to analytically in terms of elements and subsets of the Cartesian plane $\mathbb R^2$). Already in \parencite{TheBook} we find a portion of algebraic topology and homotopy theory developed in homotopy type theory (homotopy groups, including the fundamental group of the circle, the Hopf fibration, the Freudenthal suspension theorem and the van Kampen theorem, for example). Here we give an elementary construction in homotopy type theory of the real projective spaces $\rprojective{n}$ and we develop some of their basic properties.

In classical homotopy theory the real projective space $\rprojective{n}$ is either defined as the space of lines through the origin in $\mathbb R^{n+1}$ or as the quotient by the antipodal action of the $2$\nobreakdash-element group on the sphere $\Sn^n$ \cite{HatcherAT}. It has a simple cell complex description with one cell in each dimension $i \le n$. Real projective spaces are frequently used as examples or as bases for computations. Each $\rprojective{n}$ has as universal covering space the sphere $\Sn^n$, and its fundamental group is $\Zmodtwo$ for $n>1$.

In homotopy type theory, we would like to have access to these spaces and their basic properties, but the above definitions do not directly carry over to this setting. We instead give an elementary inductive construction of the spaces $\rprojective{n}$ together with their universal coverings $\mathsf{cov}^n : \rprojective{n} \to \UU$. (We write $\UU$ for the first universe of types.) Indeed, the universal coverings are fibrations where the fibers are $2$\nobreakdash-element sets, so the maps $\mathsf{cov}^n$ factor through the sub-type of $\UU$ consisting of $2$\nobreakdash-element sets, $\UU_{\Sn^0}$ (see Section~\ref{sec:UUS0} below).
To illustrate how our approach matches traditional constructions, note that we get the real projective plane by attaching a 2-cell to the circle along the map of degree of $2$.
Geometrically, this amounts to gluing a disk to M{\"o}bius strip along its boundary, cf.~Fig.~\ref{fig:mobius}.

\begin{figure}
  \centering
  \pgfplotsset{width=7cm,compat=1.14}
  \pgfplotsset{
    colormap={graywhite}{
      rgb255=(64,64,64)
      rgb255=(255,255,255)
    }
  }
  \begin{tikzpicture}
    \begin{axis}[
      hide axis,
      view = {40}{40}
    ]
    \addplot3 [
      surf,
      colormap name=graywhite,
      shader     = faceted interp,
      point meta = x,
      samples    = 40,
      samples y  = 5,
      z buffer   = sort,
      domain     = 0:360,
      y domain   =-0.5:0.5
    ] (
      {(1+0.3*y*cos(x/2)))*cos(x)},
      {(1+0.3*y*cos(x/2)))*sin(x)},
      {0.3*y*sin(x/2)}
    );
  
    \addplot3 [
      samples=50,
      domain=-145:185,
      samples y=0,
      thick
    ] (
      {cos(x)},
      {sin(x)},
      {0}
    );
    \end{axis}
  \end{tikzpicture}

  \caption{M{\"o}bius strip. The real projective plane is obtained by gluing on a disk along the boundary. (Ti{\itshape k}Z code adapted from \cite{JakeTikZ}.)}
  \label{fig:mobius}
\end{figure}
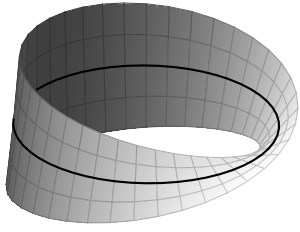

In a companion article, we shall describe another construction of both real and complex projective spaces as \emph{homotopy quotients} of ($\infty$\nobreakdash-)group actions (of the $2$\nobreakdash-element group $O(1)$ and the circle group $U(1)$, respectively).

Our construction (like other constructions in homotopy type theory) gives the \emph{homotopy types} of real projective spaces, as opposed to more refined structure such as that of smooth manifolds or real algebraic varieties.
A benefit of having a construction in homotopy type theory is that it applies in any model thereof, and not just in the standard one of infinity groupoids, $\inftyGpd$.
For example, it is conjectured (and proved in many cases) that Grothendieck $\infty$\nobreakdash-toposes provide models.
As the real projective spaces are just certain colimits built from the unit type, they sit in the discrete part of any Grothendieck $\infty$\nobreakdash-topos, i.e., in the inverse image of the geometric morphism to $\inftyGpd$.
In the settings of real- or smooth-cohesive homotopy type theory, they should be the shapes of the incarnations of the real projective spaces as topological or smooth manifolds, respectively, cf.~the similar situation for the circle discussed in~\cite{Shulman2015}.
It is also expected that homotopy type theory can be modeled in \emph{elementary} $\infty$\nobreakdash-toposes~\cite{Shulman2017}, which do not in general admit a geometric morphism to $\inftyGpd$, and hence something like our construction would be needed to access the homotopy types of the real projective spaces in such models.

In the remainder of this introduction we describe in more detail the precise setting of homotopy type theory we are working in as well as the techniques we use.

We work in an intensional dependent type theory with a univalent universe $\UU$ containing the type of natural numbers, the unit type and the empty type, and we assume that $\UU$ is closed under homotopy pushouts. Furthermore, to use the universal property of pushouts to map into arbitrary types regardless of their size, we assume function extensionality for all dependent types. (This is automatic if every type belongs to a univalent universe, but we do not assume more than one universe.) The types $A:\UU$ are called \emph{small types}. 

Let us briefly recall how homotopy pushouts are obtained as higher inductive types: if we are given $A,B,C:\UU$ and $f:C\to A$ and $g:C\to B$, then we form the pushout of $f$ and $g$
\begin{equation*}
\begin{tikzcd}
C \arrow[r,"g"] \arrow[d,swap,"f"] & B \arrow[d,"\inr"] \\
A \arrow[r,swap,"\inl"] & A+_C B
\end{tikzcd}
\end{equation*}
as the higher inductive type $A+_C B:\UU$ with point constructors
\begin{align*}
\inl & : A\to A+_C B \\
\inr & : B\to A+_C B
\intertext{and a path constructor}
\glue & : \prd{c:C} \inl(f(c))=\inr(g(c)).
\end{align*}
The \emph{elimination principle} of $A+_C B$ provides a way of defining sections of type families $P:A+_C B\to\type$ (not necessarily small). If we have
\begin{align*}
p_A & : \prd{a:A} P(\inl(a)) \\
p_B & : \prd{b:B} P(\inr(b)) 
\intertext{as well as paths}
p_C & : \prd{c:C} p_A(f(c)) =_{\jglue(c)}^P p_B(g(c)),
\end{align*}
then we get a section $f : \prd{x:A+_CB} P(x)$. Moreover, the section $f$ we obtain this way satisfies the computation rules $f(\inl(a))\jdeq p_A(a):P(\inl(a))$ and $f(\inr(b))\jdeq p_B(b) : P(\inr (b))$. We also have a witness for propositional equality for map on paths, $w : \prd{c:C} \mapdepfunc f(\jglue(c)) = p_C(c)$ (we refer to the homotopy type theory book for details on dependent paths $a=_p^Pb$ and the map on paths operations $\mapfunc f$ and $\mapdepfunc f$).

From pushouts, it is easy to derive also homotopy coequalizers of parallel maps $f,g:A\rightrightarrows B$ for $A,B:\UU$. From these we get sequential colimits of diagrams
\begin{equation*}
\begin{tikzcd}
A_0 \arrow[r,"f_0"] & A_1 \arrow[r,"f_1"] & A_2 \arrow[r,"f_2"] & \cdots
\end{tikzcd}
\end{equation*}
consisting of $A:\N\to\UU$ and $f:\prd{n:\N} A_n\to A_{n+1}$.

We also get the \emph{suspension} $\susp A$ of a type $A$ as the pushout
\begin{equation*}
\begin{tikzcd}
A \arrow[r] \arrow[d] & \unit \arrow[d,"\south"] \\
\unit \arrow[r,swap,"\north"] & \susp A.
\end{tikzcd}
\end{equation*} 
The spheres $\Sn^n$ are defined by recursion on $n:\N_{-1}$, by iteratively suspending
the $(-1)$\nobreakdash-sphere, which is the empty type. Thus, for each $n:\N_{-1}$, the $(n+1)$\nobreakdash-sphere
has a north pole and a south pole, and there is a homotopy $\glue:\Sn^n\to (\north=\south)$ which suspends the $n$\nobreakdash-sphere between the poles. In particular, the $0$\nobreakdash-sphere $\Sn^0$ is a $2$\nobreakdash-element type.

The pushout of the product projections $A \times B \to A$ and $A \times B \to B$ gives the \emph{join} $\join AB$. Lemma~8.5.10 in \cite{TheBook} establishes that $\susp A \simeq \join{\Sn^0}A$, so it follows that $\Sn^{n+1} \simeq \join{\Sn^0}{\Sn^n}$.

With pushouts, it is also possible to define for any type $A:\UU$ its propositional truncation $\sbrck A$, which is the reflection of $A$ into the type of mere propositions.
This means we have a function ${\lvert{{-}}\rvert} : A \to \sbrck A$, precomposition with which gives an equivalence $(\sbrck A \to B) \simeq (A \to B)$ for any mere proposition $B$.
There are several constructions of the propositional truncation in terms of pushouts and colimits due to Van Doorn~\cite{vanDoorn2016}, Kraus~\cite{Kraus2016}, and the second-named author~\cite{joinconstruction}.

As a final special case of a pushout we shall need, we have the \emph{mapping cone} $C_f$ of a map $f : A \to B$. This is just the pushout of $f$ and the unique map $A \to 1$. In case $A \equiv \Sn^n$ is a sphere, we think of $f$ as describing an $n$\nobreakdash-sphere in $B$ which is bounded by an $(n+1)$\nobreakdash-cell in $C_f$. We also say that $C_f$ is obtained by attaching an $(n+1)$\nobreakdash-cell to $B$ via the attaching map $f$. The inclusion of the unique element of $1$ can be thought of as a \emph{hub} and the glue paths parametrized by elements of $\Sn^n$ as \emph{spokes}. (In Section~6.7 of \cite{TheBook} the term `hubs-and-spokes method' is used for this way of attaching cells.)

A model for our setup is given by the cubical sets of \cite{CCHM2016} (they verify the rules for suspensions---the same argument works for pushouts), and this model can itself be interpreted in extensional Martin-L\"of type theory with one universe containing $0,1,\N$ and closed under disjoint union and dependent sums and products.%
\footnote{The model has been formalized in Nuprl by Mark Bickford: \url{http://www.nuprl.org/wip/Mathematics/cubical!type!theory/}}

The remainder of this paper is organized as follows. In Section~\ref{sec:UUS0} we study the type of $2$\nobreakdash-element sets, $\UU_{\Sn^0}$, which is a (large) model of $\rprojective\infty$. In particular, we show that $\eqv{(\Sn^0=A)}{A}$ for any $A:\UU_{\Sn^0}$, and that $\UU_{\Sn^0}$ classifies the $\Sn^0$\nobreakdash-bundles over any type. To prove these facts, we use Licata's encode-decode method, in the following form. 

\begin{lem}[Encode-decode method]\label{lem:encode-decode}
Let $A$ be a pointed type with base point $a_0$, and let $B:A\to\UU$ be a type
family with a point $b_0:B(a_0)$. Then the following are equivalent:
\begin{enumerate}[label={\normalfont(}\roman*\/{\normalfont)},itemsep=0.25em]
\item The type $\sm{x:A}B(x)$ is contractible.
\item The fiberwise map
\begin{equation*}
\mathsf{enc}_{a_0,b_0}:\prd*{x:A} (a_0=x)\to B(x)
\end{equation*}
defined by $\mathsf{enc}_{a_0,b_0}(\refl{a_0})\defeq b_0$ is an equivalence.
\end{enumerate}
\end{lem}

\begin{proof}
Since the total space $\sm{x:A} a_0=x$ is contractible, this is a direct corollary of Theorem 4.7.7 of \cite{TheBook}.
\end{proof}

Then in Section~\ref{sec:fdrp} we use this to give definitions of $\rprojective n$ for finite $n$, together with its tautological $\Sn^0$\nobreakdash-bundle $\mathsf{cov}^n_{\Sn^0}:\rprojective n\to\UU_{\Sn^0}$. To recover the description of $\rprojective{n+1}$ as $\rprojective n$ with an $(n+1)$\nobreakdash-cell attached to it, we show that the total space of the tautological bundle of $\rprojective{n}$ is the $n$\nobreakdash-sphere. Here we shall need the flattening lemma for pushouts, cf.~\cite[Lemma~8.5.3]{TheBook}.

\begin{lem}[Flattening Lemma]\label{lem:flattening}
Consider a pushout square
\begin{equation*}
\begin{tikzcd}
A \arrow[r,"g"] \arrow[d,swap,"f"] & Y \arrow[d,"\inr"] \\
X \arrow[r,swap,"\inl"] & X+_A Y
\end{tikzcd}
\end{equation*}
and let $P_X:X\to\UU$ and $P_Y:Y\to \UU$ be type families that are compatible
in the sense that there is a fiberwise equivalence $e:\prd{a:A}\eqv{P_X(f(a))}{P_Y(g(a))}$. 

Let $P_{tot}:X+_A Y \to \UU$ be the unique type family, defined via the universal
property of $X+_A Y$, for which there are equivalences
\begin{align*}
\alpha_\inl : & \prd{x:X} \eqv{P_X(x)}{P_{tot}(\inl(x))} \\
\alpha_\inr : & \prd{y:Y} \eqv{P_Y(y)}{P_{tot}(\inr(y))}.
\end{align*}
Then the square
\begin{equation*}
\begin{tikzcd}[column sep=7em]
\sm{a:A}P_X(f(a)) \arrow[d,swap,"\pairr{a,p}\mapsto\pairr{f(a),p}"] \arrow[r,"\pairr{a,p}\mapsto\pairr{g(a),e(p)}"] & \sm{y:Y}P_Y(y) \arrow[d,"\pairr{y,p}\mapsto\pairr{\inr(y),\alpha_\inr(p)}"'] \\
\sm{x:X}P_X(x) \arrow[r,swap,"\pairr{x,p}\mapsto\pairr{\inl(x),\alpha_\inl(p)}"] & \sm{z:X+_A Y} P_{tot}(z)
\end{tikzcd}
\end{equation*}
of which commutativity is witnessed by
\begin{equation*}
\lam{\pairr{a,p}} \pairr{\glue(a),\blank},
\end{equation*}
is a pushout square.
\end{lem}

As a corollary, we obtain the fiber sequence
\begin{equation*}
\begin{tikzcd}
\Sn^0 \arrow[r,hook] & \Sn^n \arrow[r,->>] & \rprojective{n}.
\end{tikzcd}
\end{equation*}
By the fiber sequence notation $F\hookrightarrow E\twoheadrightarrow B$ we mean that we have a dependent type over $B$, such that the fiber over the base point of $B$ is (equivalent to) $F$, and the total space is (equivalent to) $E$.

Finally, in Section~\ref{sec:idrp} we recover $\rprojective\infty$ as the sequential colimit of the $\rprojective n$, and this is now a small type. We show that the tautological bundle of $\rprojective\infty$ is an equivalence into $\UU_{\Sn^0}$, so we see that $\UU_{\Sn^0}$ is indeed a model for $\rprojective{\infty}$. Section~\ref{sec:conclusion} concludes.

\section{The type of $2$-element sets}
\label{sec:UUS0}

\begin{defn}
For any type $X$, we define \define{the connected component of $X$ in $\UU$} to be the type \[\UU_X\defeq\sm{A:\UU}\sbrck{A=X}.\] In particular, we have the type \[\UU_{\Sn^{0}}\jdeq\sm{A:\UU}\sbrck{A={\Sn^{0}}}\] of $2$\nobreakdash-element sets.

An \define{$X$\nobreakdash-bundle} over a type $A$ is defined to be a type family $B:A\to\UU_X$. 
\end{defn}

A term of type $\UU_X$ is formally a pair of a small type $A:\UU$ together with a term of type $\sbrck{A=X}$, but since the latter is a mere proposition we usually omit it, and consider the term itself as a small type.

\begin{thm}\label{thm:ptd_2elt_sets}
The type
\begin{equation*}
\sm{A:\UU_{\Sn^{0}}}A
\end{equation*}
of pointed $2$\nobreakdash-element sets is contractible.
\end{thm}

Since the theorem is a statement about pointed $2$\nobreakdash-element sets, 
we will invoke the following general lemma which computes equality of pointed types.

\begin{lem}\label{lem:equiv_of_ptdtype}
For any $A,B:\UU$ and any $a:A$ and $b:B$, we have an equivalence of type
\begin{equation*}
\eqv{\Big(\pairr{A,a}=\pairr{B,b}\Big)}{\Big(\sm{e:\eqv{A}{B}}e(a)=b\Big)}.
\end{equation*}
\end{lem}

\begin{proof}[Construction]
By Theorem 2.7.2 of \cite{TheBook}, the type on the left hand side is
equivalent to the type $\sm{p:A=B}\trans{p}{a}=b$.
(As in \cite{TheBook}, $p_*$ denotes the transport function relative to a 
fibration.)
By the univalence axiom, the map 
\begin{equation*}
\mathsf{idtoequiv}_{A,B}:(A=B)\to (\eqv{A}{B})
\end{equation*}
is an equivalence for each $B:\UU$. 
Therefore, we have an equivalence of type
\begin{equation*}
\eqv{\Big(\sm{p:A=B}\trans{p}{a}=b\Big)}{\Big(\sm{e:\eqv{A}{B}}\trans{\mathsf{equivtoid}(e)}{a}=b\Big)}
\end{equation*} 
Moreover, by equivalence induction (the analogue of path induction for 
equivalences), we can compute the transport:
\begin{equation*}
\trans{\mathsf{equivtoid}(e)}{a}=e(a).
\end{equation*}
It follows that $\eqv{(\trans{\mathsf{equivtoid}(e)}{a}=b)}
{(e(a)=b)}$.
\end{proof}

Furthermore, we will invoke the following general lemma which computes equality of pointed
equivalences.

\begin{lem}\label{lem:equiv_of_ptdequiv}
For any $\pairr{e,p},\pairr{f,q}:\sm{e:\eqv{A}{B}}e(a)=b$, we have an equivalence of type
\begin{equation*}
\eqv{\Big(\pairr{e,p}=\pairr{f,q}\Big)}{\Big(\sm{h:e\htpy f} p=\ct{h(a)}{q}\Big)}.
\end{equation*}
\end{lem}

\begin{proof}[Construction]
The type $\pairr{e,p}=\pairr{f,q}$ is equivalent
to the type $\sm{h:e=f}\trans{h}{p}=q$.
Note that by the principle of function extensionality,
the map $\mathsf{idtohtpy}:(e=f)\to(e\htpy f)$
is an equivalence. Furthermore, it follows by homotopy induction that for any 
$h:e\htpy f$ we have an equivalence of type
\begin{equation*}
\eqv{(\trans{\mathsf{htpytoid}(h)}{p}=q)}
    {(p= \ct{h(a)}{q})}.\qedhere
\end{equation*}
\end{proof}

We are now ready to prove that the type of pointed $2$\nobreakdash-element sets is contractible.

\begin{proof}[Proof of \autoref{thm:ptd_2elt_sets}]
We take $\pairr{\Sn^0,\north}$ as the center of contraction. We need to define
an identification of type $\pairr{\Sn^0,\north}=\pairr{A,a}$, 
for any $A:\UU_{\Sn^0}$ and $a:A$. 

Let $A:\UU_{\Sn^0}$ and $a:A$. By \autoref{lem:equiv_of_ptdtype}, we have
an equivalence of type
\begin{equation*}
\eqv{\Big(\pairr{\Sn^0,\north}=\pairr{A,a}\Big)}{\Big(\sm{e:\eqv{\Sn^0}{A}}e(\north)=a\Big)}.
\end{equation*}
Hence we can complete the proof by constructing a term of type 
\begin{equation}\label{eq:Sn0_ptdequiv}
\sm{e:\eqv{\Sn^0}{A}}e(\north)=a.
\end{equation} 

It is time for a little trick. Instead of constructing a term the type in
\autoref{eq:Sn0_ptdequiv}, we will show that this type is contractible.
Since being contractible is a mere proposition, 
this allows us to eliminate the assumption $\sbrck{\Sn^0=A}$
into the assumption $p:\Sn^0=A$. Note that the end point of $p$ is free.
Therefore we eliminate $p$ into $\refl{\Sn^0}$. 
Thus, we see that it suffices to show that the type
\begin{equation*}
\sm{e:\eqv{\Sn^0}{\Sn^0}}e(\north)=a
\end{equation*}
is contractible for any $a:\Sn^0$. 

This can be done by case analysis on $a:\Sn^0$. Since we have the equivalence
$\mathsf{neg}:\eqv{\Sn^0}{\Sn^0}$ that swaps $\north$ and $\south$, it follows
that $\sm{e:\eqv{\Sn^0}{\Sn^0}}e(\north)=\north$ is contractible if and only
if $\sm{e:\eqv{\Sn^0}{\Sn^0}}e(\north)=\south$ is contractible. Therefore, we
only need to show that the type
\begin{equation*}
\sm{e:\eqv{\Sn^0}{\Sn^0}}e(\north)=\north
\end{equation*}
is contractible. 
For the center of contraction we take $\pairr{\idfunc[\Sn^0],\refl{\north}}$.
It remains to construct a term of type
\begin{equation*}
\prd{e:\eqv{\Sn^0}{\Sn^0}}{p:e(\north)=\north} \pairr{e,p}=\pairr{\idfunc[\Sn^0],\refl{\north}}.
\end{equation*} 

Let $e:\eqv{\Sn^0}{\Sn^0}$ and $p:e(\north)=\north$.
By \autoref{lem:equiv_of_ptdequiv}, we have an equivalence of type
\begin{equation*}
\eqv{\Big(\pairr{e,p}=\pairr{\idfunc[\Sn^0],\refl{\north}}\Big)}
    {\Big(\sm{h:e\htpy \idfunc[\Sn^0]} p=h(\north)\Big)}.
\end{equation*}
Hence it suffices
to construct a term of the type on the right hand side.

We define a homotopy $h:e\htpy \idfunc[\Sn^0]$ by case analysis: we take
$h(\north)\defeq p$. To define $h(\south)$, note that the type
$\hfib{e}{\south}$ is contractible. Therefore, we have a center of contraction
$\pairr{x,q}:\hfib{e}{\south}$. Recall that equality on $\Sn^0$ is decidable,
so we have a term of type $(x=\north)+(x=\south)$. Since $e(\north)=\north$,
it follows that $\neg(x=\north)$. Therefore we have $x=\south$ and $e(x)=\south$.
It follows that $e(\south)=\south$, which we use to define $h(\south)$. 
\end{proof}

The main application we have in mind for \autoref{thm:ptd_2elt_sets}, is
a computation of the identity type of the type of $2$\nobreakdash-element sets, via
the encode-decode method, \autoref{lem:encode-decode}.

\begin{cor}\label{cor:id_U2}
The canonical map
\begin{equation*}
\mathsf{enc}_{{\Sn^{0}},\north} : \prd*{A:\UU_{\Sn^{0}}} ({\Sn^{0}}= A)\to A
\end{equation*}
is an equivalence.
\end{cor}

Another way of stating the following theorem, is by saying that the map
$\unit\to\UU_{\Sn^{0}}$ \emph{classifies} the ${\Sn^{0}}$\nobreakdash-bundles.

\begin{thm}\label{lem:classifyer_U2}
Let $B:A\to\UU_{\Sn^{0}}$ be a ${\Sn^{0}}$\nobreakdash-bundle. Then the square
\begin{equation*}
\begin{tikzcd}
\sm{x:A}B(x) \arrow[r] \arrow[d,swap,"\proj 1"] & \unit \arrow[d,"{\Sn^{0}}"] \\
A \arrow[r,swap,"B"] & \UU_{\Sn^{0}}
\end{tikzcd}
\end{equation*}
commutes via a homotopy $R_{A,B}:\prd{x:A}{y:B(x)} \eqv{B(x)}{\Sn^0}$, and is a pullback square. 
\end{thm}

\begin{proof}[Construction]
Since $B(x):\UU_{\Sn^{0}}$ for any $x:A$, 
we have by \autoref{cor:id_U2} the fiberwise equivalence
\begin{equation*}
\mathsf{enc}_{\Sn^0,\N}(B(x)):\eqv{(\Sn^0=B(x))}{B(x)}
\end{equation*} 
indexed by $x:A$. 
Hence it follows by Theorem 4.7.7 of \cite{TheBook} that the induced map
of total spaces is an equivalence. It follows that the diagram
\begin{equation*}
\begin{tikzcd}
\sm{x:A}B(x) \arrow[drr,bend left=15] \arrow[ddr,bend right=15,swap,"\proj 1"] \arrow[dr,densely dotted,"\eqvsym"] \\
& \sm{x:A} ({\Sn^{0}}=B(x)) \arrow[r] \arrow[d,swap,"\proj 1"] & \unit \arrow[d,"{\lam{\nameless}\Sn^{0}}"] \\
& A \arrow[r,swap,"B"] & \UU_{\Sn^{0}}
\end{tikzcd}
\end{equation*}
commutes. Since the inner square is a pullback square, it follows that the outer square is a pullback square.
\end{proof}

\section{Finite dimensional real projective spaces}
\label{sec:fdrp}

Classically, the $(n+1)$-st real projective space can be obtained by attaching an $(n+1)$-cell to the $n$-th real projective space. This suggests a way of defining the real projective spaces that involves simultaneously defining $\rprojective{n}$ and an attaching map $\alpha_n : \Sn^n\to\rprojective{n}$. Then we obtain $\rprojective{n+1}$ as the mapping cone of $\alpha_n$, i.e., as a pushout
\begin{equation*}
\begin{tikzcd}
\Sn^n \arrow[r,"\alpha_n"] \arrow[d] & \rprojective{n} \arrow[d] \\
\unit \arrow[r] & \rprojective{n+1},
\end{tikzcd}
\end{equation*}
and we have to somehow find a way to define the attaching map $\alpha_{n+1}:\Sn^{n+1}\to\rprojective{n}$ to continue the inductive procedure.
However, it is somewhat tricky to obtain these attaching maps directly, and we have chosen to follow a closely related path towards the definition of the real projective spaces that takes advantage of the machinery of dependent type theory. 

Observe that the attaching map $\alpha_n:\Sn^n\to\rprojective{n}$ is just the tautological bundle (or the quotient map that identifies the antipodal points). This suggests that we may proceed by defining simultaneously the real projective space $\rprojective{n}$ and its tautological bundle $\mathsf{cov}^n_{\Sn^0}$. The tautological bundle on $\rprojective{n}$ is an $\Sn^0$-bundle, so it can be described as a map $\rprojective{n}\to\UU_{\Sn^0}$. We perform this construction in \autoref{defn:realprojective} using the properties of the type of 2-element types developed in \autoref{sec:UUS0}, and in \autoref{thm:Sn_totalcov} we show that the total space of the tautological bundle on $\rprojective{n}$ is the $n$-sphere. 

\begin{defn}\label{defn:realprojective}
We define simultaneously for each $n:\N_{-1}$, 
the \define{$n$-dimensional real projective space} $\rprojective{n}$, 
and the \define{tautological bundle} $\mathsf{cov}^n_{\Sn^{0}}:\rprojective{n}\to \UU_{\Sn^{0}}$.
\end{defn}

\begin{proof}[Construction]
The construction is by induction on $n:\N_{-1}$.
For the base case $n\defeq -1$, 
we take $\rprojective{-1}\defeq\emptyt$. 
Then there is a unique map of type $\rprojective{-1}\to \UU_{\Sn^{0}}$, which we
take as our definition of $\mathsf{cov}_{\Sn^{0}}^{-1}$.

For the inductive step, suppose $\rprojective{n}$ and $\mathsf{cov}^n_{\Sn^{0}}$ are defined. Then we define $\rprojective{n+1}$ to be the pushout
\begin{equation*}
\begin{tikzcd}
\sm{x:\rprojective{n}}\mathsf{cov}^n_{\Sn^{0}}(x) \arrow[d,swap,"\proj 1"] \arrow[r] & \unit \arrow[d,"\base"] \\
\rprojective{n} \arrow[r,swap,"\mathsf{incl}"] & \rprojective{n+1}
\end{tikzcd}
\end{equation*}
In other words, 
$\rprojective{n+1}$ is the \emph{mapping cone} of the tautological bundle, 
when we view the tautological bundle as the projection 
$\proj 1:(\sm{x:\rprojective{n}}\mathsf{cov}_{\Sn^0}^n(x))\to\rprojective{n}$. 

To define $\mathsf{cov}^{n+1}_{\Sn^{0}}:\rprojective{n+1}\to \UU_{\Sn^{0}}$
we use the universal property of $\rprojective{n+1}$. 
Therefore, it suffices to show that the outer square in the diagram
\begin{equation}\label{eq:diagram}
\begin{tikzcd}
\sm{x:\rprojective{n}}\mathsf{cov}^n_{\Sn^{0}}(x) \arrow[d,swap,"\proj 1"] \arrow[r] & \unit \arrow[d,swap,"\base"] \arrow[ddr,bend left=15,"\Sn^0"]\\
\rprojective{n} \arrow[drr,bend right=15,swap,"\mathsf{cov}^n_{\Sn^{0}}"] \arrow[r,swap,"\mathsf{incl}"] & \rprojective{n+1} \arrow[dr,densely dotted] \\
& & \UU_{\Sn^{0}}
\end{tikzcd}
\end{equation}
commutes. Indeed, in \autoref{lem:classifyer_U2} we have constructed a homotopy 
\begin{equation*}
R_n\defeq R_{\rprojective{n},\mathsf{cov}^n_{\Sn^0}}:\prd{x:\rprojective{n}}{y:\mathsf{cov}^n_{\Sn^0}} \eqv{\mathsf{cov}^n_{\Sn^0}(x)}{\Sn^0},
\end{equation*}
and in fact, this square is a pullback.
\end{proof}

\begin{eg}
We have $\rprojective{-1}=\emptyt$, $\rprojective{0}=\unit$, and $\rprojective{1}=\Sn^1$. 
\end{eg}

\begin{thm}\label{thm:Sn_totalcov}
For each $n:\N_{-1}$, there is an equivalence
\begin{equation*}
e_n:\eqv{\Sn^n}{\sm{x:\rprojective{n}}\mathsf{cov}^n_{\Sn^{0}}(x)}.
\end{equation*}
\end{thm}

In other words, $\rprojective{n+1}$ is obtained from $\rprojective{n}$ by attaching a single $(n+1)$\nobreakdash-disk, i.e., as a pushout
\begin{equation*}
\begin{tikzcd}
\Sn^n \arrow[r] \arrow[d,swap,"\proj1\circ e_n"] & \unit \arrow[d] \\
\rprojective{n} \arrow[r] & \rprojective{n+1}.
\end{tikzcd}
\end{equation*}

\begin{proof}
For $n\jdeq -1$, we have $\rprojective{-1}\jdeq\emptyt$ and the unique tautological bundle $\mathsf{cov}^{-1}_{\Sn^{0}}$. Therefore the type $\sm{x:\rprojective{-1}}\mathsf{cov}^{-1}_{\Sn^0}(x)$ is equivalent to the empty type, which is $\Sn^{-1}$ by definition. This gives the base case.

Now assume that we have an equivalence $e_n:\eqv{\Sn^n}{\sm{x:\rprojective{n}}\mathsf{cov}^n_{\Sn^{0}}(x)}$. 
Our goal is to construct the equivalence
\begin{equation*}
e_{n+1}:\eqv{\Sn^{n+1}}{\sm{x:\rprojective{n+1}}\mathsf{cov}^{n+1}_{\Sn^{0}}(x)}.
\end{equation*}
such that the square
\begin{equation}\label{eq:Sn_totalcov_natural}
\begin{tikzcd}
\Sn^{n} \arrow[d,swap,"e_{n}"] \arrow[r,"\inl"] & \Sn^{n+1} \arrow[d,"e_{n+1}"] \\
\sm{x:\rprojective{n}}\mathsf{cov}_{\Sn^0}^n(x) \arrow[r] & \sm{x:\rprojective{n+1}}\mathsf{cov}_{\Sn^0}^{n+1}(x)
\end{tikzcd}
\end{equation}
commutes. By the functoriality of the join (or equivalently, by equivalence induction on $e_n$), it suffices to find an equivalence
\begin{equation*}
\alpha:\eqv{\join{\Big(\sm{x:\rprojective{n}}\mathsf{cov}_{\Sn^0}^n(x)\Big)}{\Sn^0}}{\sm{x:\rprojective{n+1}}\mathsf{cov}_{\Sn^0}^{n+1}(x)},
\end{equation*}
such that the bottom triangle in the diagram
\begin{equation*}
\begin{tikzcd}
\Sn^n \arrow[r,"\inl"] \arrow[d,swap,"e_n"] & \Sn^{n+1} \arrow[d,"\join{e_n}{\idfunc[\Sn^0]}"] \\
\sm{x:\rprojective{n}}\mathsf{cov}_{\Sn^0}^n(x) \arrow[r,"\inl"] \arrow[dr] & \join{\Big(\sm{x:\rprojective{n}}\mathsf{cov}_{\Sn^0}^n(x)\Big)}{\Sn^0} \arrow[d,"\alpha"] \\
& \sm{x:\rprojective{n+1}}\mathsf{cov}_{\Sn^0}^{n+1}(x)
\end{tikzcd}
\end{equation*}
commutes.
We construct this equivalence using the flattening lemma, \autoref{lem:flattening}, from which we get a pushout square:
\begin{equation*}
\begin{tikzcd}[column sep=0.5em]
\sm{x:\rprojective{n}}{y:\mathsf{cov}_{\Sn^0}^n(x)}\mathsf{cov}_{\Sn^0}(x) \arrow[r] \arrow[d] & \sm{t:\unit}\Sn^0 \arrow[d] \\
\sm{x:\rprojective{n}}\mathsf{cov}_{\Sn^0}^n(x) \arrow[r] & \sm{x:\rprojective{n+1}}\mathsf{cov}_{\Sn^0}^{n+1}(x)
\end{tikzcd}
\end{equation*}
We can calculate this pushout by constructing a natural transformation of spans (diagrams in $\UU$ of the form $\cdot\leftarrow\cdot\rightarrow\cdot$), as indicated by the diagram in Fig.~\ref{fig:sphere-equiv}.
\begin{figure*}
  \centering
\begin{tikzcd}[column sep=6em]
\sm{x:\rprojective{n}}\mathsf{cov}_{\Sn^0}^n(x) \arrow[d,swap,"\idfunc"]
  & \sm{x:\rprojective{n}}{y:\mathsf{cov}_{\Sn^0}^n(x)}\mathsf{cov}_{\Sn^0}(x) \arrow[l,swap,"\pairr{x,z}\mapsfrom\pairr{x,y,z}" yshift=1ex] \arrow[d,densely dotted,"u"] \arrow[r,"\pairr{x,y,z}\mapsto\pairr{\ttt,R_n(x,y,z)}" yshift=1ex] 
  & \sm{t:\unit}\Sn^0 \arrow[d,"\proj 2"] \\
\sm{x:\rprojective{n}}\mathsf{cov}_{\Sn^0}^n(x)
  &
\Big(\sm{x:\rprojective{n}}\mathsf{cov}_{\Sn^0}^n(x)\Big)\times\Sn^0 \arrow[l,"\pi_1"] \arrow[r,swap,"\pi_2"]
  & \Sn^0
\end{tikzcd}
\caption{Map of spans used in the proof of Thm~\ref{thm:Sn_totalcov}. The map $u$ is given by $\pairr{x,y,z}\mapsto\pairr{x,z,R_n(x,y,z)}$.}
\label{fig:sphere-equiv}
\end{figure*}
To show that the map $u$ in Fig.~\ref{fig:sphere-equiv} is an equivalence, it suffices to show that $R_n(x,y,z)=R_n(x,y,z)$ for any $x$, $y$, and $z$, because then it follows that $u$ is homotopic to the total map of a fiberwise equivalence. More generally, it suffices to show that $R_{\UU_{\Sn^0},T}(X,x,y)=R_{\UU_{\Sn^0},T}(X,y,x)$, where $T$ is the tautological bundle on $\UU_{\Sn^0}$. Since $\UU_{\Sn^0}$ is connected and since our goal is a mere proposition, we only need to verify the claim at the base point $\Sn^0$ of $\UU_{\Sn^0}$. This boils down to verifying that the group multiplication of $\Zmodtwo$ is indeed commutative.
\end{proof}

\begin{cor}
We obtain the fiber sequence
\begin{equation*}
\begin{tikzcd}
\Sn^0 \arrow[r,hook] & \Sn^n \arrow[r,->>] & \rprojective{n}.
\end{tikzcd}
\end{equation*}
Hence, for each $k\geq 2$ we have $\pi_k(\Sn^n)=\pi_k(\rprojective{n})$. 
\end{cor}

\begin{proof}
Since we have the double cover $\mathsf{cov}^n_{\Sn^{0}}:\rprojective{n}\to\UU_{\Sn^{0}}$ with total space $\Sn^n$, we obtain the long exact sequence
\begin{equation*}
\begin{tikzcd}
  \cdots \arrow[r]
  & \pi_{k+1}(\Sn^n) \arrow[r] \arrow[d, phantom, ""{coordinate, name=Z}]
  & \pi_{k+1}(\rprojective n) \arrow[dll, rounded corners,
      to path={ -- ([xshift=8.5ex]\tikztostart.center)
                |- (Z) [near end]\tikztonodes
                -| ([xshift=-8ex]\tikztotarget.center) -- (\tikztotarget)}] \\
  \pi_k(\Sn^0) \arrow[r]
  & \pi_k(\Sn^n) \arrow[r] \arrow[d, phantom, ""{coordinate, name=W}]
  & \pi_k(\rprojective n) \arrow[dll, rounded corners,
      to path={ -- ([xshift=8.5ex]\tikztostart.center)
                |- (W) [near end]\tikztonodes
                -| ([xshift=-8ex]\tikztotarget.center) -- (\tikztotarget)}] \\
  \pi_{k-1}(\Sn^0) \arrow[r]
  & \pi_{k-1}(\Sn^n) \arrow[r]
  & \cdots
\end{tikzcd}
\end{equation*}
Since $\pi_k({\Sn^{0}})=0$ for $k\geq 1$, we get the desired isomorphisms.
\end{proof}

\section{The infinite dimensional real projective space}
\label{sec:idrp}

Observe that from the definition of $\rprojective{n}$ and its tautological 
cover, we obtain a commutative diagram of the form:
\begin{equation*}
\begin{tikzcd}[row sep=large,column sep=large]
\rprojective{-1} \arrow[r,"\mathsf{incl}"] \arrow[dr,swap,"\mathsf{cov}_{\Sn^0}^{-1}"] 
& \rprojective{0} \arrow[d,swap,near start,"\mathsf{cov}_{\Sn^0}^{0}"] \arrow[r,"\mathsf{incl}"] 
& \rprojective{1} \arrow[dl,swap,"\mathsf{cov}_{\Sn^0}^{1}"] \arrow[r,"\mathsf{incl}"] 
& \cdots \arrow[dll,"\mathsf{cov}_{\Sn^0}^{2}"]\\
& \UU_{\Sn^0}
\end{tikzcd}
\end{equation*}
Using this sequence, we define the infinite dimensional real projective space
and its tautological cover:

\begin{defn}
We define the \define{infinite real projective space} $\rprojective{\infty}$ to be the sequential colimit of the finite real projective spaces. The double covers on $\rprojective{n}$ define a cocone on the type sequence of real projective spaces, so we also obtain $\mathsf{cov}^\infty_{\Sn^{0}}:\rprojective{\infty}\to \UU_{\Sn^{0}}$. 
\end{defn}

\begin{thm}\label{thm:RPoo_US0}
The double cover $\mathsf{cov}^\infty_{\Sn^{0}}$ is an equivalence from $\rprojective{\infty}$ to $\UU_{\Sn^{0}}$. 
\end{thm}

\begin{proof}
We have to show that the fibers of $\mathsf{cov}_{\Sn^0}^\infty$ are contractible.
Since being contractible is a mere proposition, and since the type $\UU_{\Sn^0}$
is connected, it suffices to show that the fiber
\begin{equation*}
\sm{x:\rprojective{\infty}}\Sn^{0}=\mathsf{cov}^\infty_{\Sn^{0}}(x)
\end{equation*}
of $\mathsf{cov}_{\Sn^0}^\infty$ at $\Sn^0:\UU_{\Sn^0}$ is contractible.
By \autoref{cor:id_U2} we have an equivalence of type
\begin{equation*}
\eqv{({\Sn^{0}}=\mathsf{cov}^\infty_{\Sn^{0}}(x))}{\mathsf{cov}_{\Sn^0}^\infty(x)},
\end{equation*}
for every $x:\rprojective{\infty}$. 
Therefore it is equivalent to show that the type
\begin{equation*}
\sm{x:\rprojective{\infty}}\mathsf{cov}^\infty_{\Sn^{0}}(x)
\end{equation*}
is contractible. The general version of the flattening lemma, as stated in
Lemma 6.12.2 in \cite{TheBook}, can be adapted for sequential colimits, so
we can pull the colimit out: it suffices to prove that
\begin{equation*}
\tfcolim_n\bigl(\sm{x:\rprojective{n}}\mathsf{cov}^n_{\Sn^{0}}(x)\bigr)
\end{equation*}
is contractible. 
To do this, observe that the equivalences of \autoref{thm:Sn_totalcov} form
a natural equivalence of type sequences as shown in Fig.~\ref{fig:type-sequences}.
\begin{figure*}
  \centering
\begin{tikzcd}
\sm{x:\rprojective{-1}}\mathsf{cov}_{\Sn^0}^{-1}(x) \arrow[r] \arrow[d,swap,"\eqvsym"]
& \sm{x:\rprojective{0}}\mathsf{cov}_{\Sn^0}^{0}(x) \arrow[r] \arrow[d,swap,"\eqvsym"]
& \sm{x:\rprojective{1}}\mathsf{cov}_{\Sn^0}^{1}(x) \arrow[r] \arrow[d,swap,"\eqvsym"]
& \cdots \\
\Sn^{-1} \arrow[r] 
& \Sn^0 \arrow[r]
& \Sn^1 \arrow[r]
& \cdots
\end{tikzcd}
\caption{Natural equivalence of type sequences for Thm~\ref{thm:RPoo_US0}.}
\label{fig:type-sequences}
\end{figure*}
Indeed, the naturality follows from \autoref{eq:Sn_totalcov_natural}.

Thus, the argument comes down to showing that $\Sn^\infty\defeq\tfcolim_n(\Sn^n)$
is contractible. This was first shown in homotopy type theory by Brunerie, and
the argument is basically that the sequential colimit of a type sequence of
strongly constant maps (viz., maps factoring through $\unit$) is always contractible.
\end{proof}

\begin{rmk}
Note that by our assumption that the universe is closed under pushouts, it
follows that each $\rprojective{n}$ is in $\UU$. 
Since the universe contains a natural numbers object $\N$, 
it also follows that the universe is closed under sequential colimits,
and therefore we have $\rprojective{\infty}:\UU$. 
Whereas a priori it is not clear that $\UU_{\Sn^0}$ is equivalent to a 
$\UU$-small type, this fact is contained in \autoref{thm:RPoo_US0}.
\end{rmk}

\section{Formalization}
\label{sec:formalization}

The results of \autoref{sec:UUS0,sec:fdrp} have been formalized in the Lean proof assistant \cite{Moura2015}.%
\footnote{The code is available at: \url{https://github.com/cmu-phil/Spectral/blob/master/homotopy/realprojective.hlean}}
The formalization follows the informal development very closely, except that we do not use directly the equivalence from \autoref{cor:id_U2}.
This is because it takes too much memory in this case to verify the symmetry property needed for \autoref{thm:Sn_totalcov}.
Instead, we observe that having decidable equality is a mere proposition, and hence every $A : \UU_{\Sn^0}$ has decidable equality.
Thus we can define $\alpha:\prd*{A:\UU_{\Sn^{0}}} A \to A \to \Sn^0$ by setting $\alpha(x,y)=\north$ if and only if $x=y$.
It is then easy to check that each $\alpha(x,\blank)$ is an involution and hence an equivalence, and $\alpha$ is clearly symmetric.

The results of \autoref{sec:idrp} require the flattening lemma for sequential colimits, which has not yet been formalized in Lean.

\section{Conclusion}
\label{sec:conclusion}

In the present article, we have constructed the real projective spaces,
both finite and infinite dimensional, using only homotopy theoretic methods. 
We have used the univalence axiom in one place, in proving that $(\Sn^0=A)\eqvsym A$
for any $A:\UU_{\Sn^0}$. Otherwise, our methods only involve the universal properties
of (homotopy) pushouts and sequential colimits. This suggests that it could be
possible to mimic our construction of the real projective spaces in arbitrary
$\infty$\nobreakdash-toposes, even though a description of the structure of a model of univalent type theory does not (yet) exist for an arbitrary $\infty$\nobreakdash-topos.

Since the outer square of the diagram in \autoref{eq:diagram}
is a pullback square, we have a few more remarks about our construction of the
real projective spaces. One could define the join of two maps $f:A\to X$ and
$g: B\to X$ with a common codomain $X$, as the pushout of the pullback. 
In the case of the real projective spaces, we are concerned with the map
$pt:\unit\to \UU_{\Sn^0}$ pointing to the type $\Sn^0$. 
By iteratively joining $pt$ with itself, we obtain the finite dimensional real
projective spaces as domains of the finite join-powers $pt^{\ast n}$. Indeed,
the tautological bundles are the maps $pt^{\ast n}$. 
This observation connects our construction of the real projective spaces with
the so-called `join construction', see \cite{joinconstruction}. 
This can be seen as a procedure of taking $\infty$-quotients in homotopy type theory. 

Finally, we mention that a special case of the join construction can also be
used to define the complex projective spaces in homotopy type theory, and to
perform Milnor's construction of the universal bundle of a (topological) group.
We hope that these methods can also be used to define the higher dimensional
Grassmannians in homotopy type theory.

\subsection*{Acknowledgments}

The authors gratefully acknowledge the support of the Air Force Office
of Scientific Research through MURI grant FA9550-15-1-0053. Any
opinions, findings and conclusions or recommendations expressed in
this material are those of the authors and do not necessarily reflect
the views of the AFOSR.

\phantomsection
\addcontentsline{toc}{section}{References}%
\balance
\printbibliography

\end{document}